\documentclass{amsart}

\usepackage{amsbsy, amsmath, amsfonts, braket, latexsym, amsthm, amsxtra, amssymb, amscd, graphics, mathrsfs, verbatim}

\usepackage[usenames,dvipsnames]{xcolor}

\usepackage[shortlabels]{enumitem}
\SetEnumerateShortLabel{a}{\textup{(\alph*)}}
\SetEnumerateShortLabel{A}{\textup{(\Alph*)}}
\SetEnumerateShortLabel{1}{\textup{(\arabic*)}}
\SetEnumerateShortLabel{i}{\textup{(\roman*)}}
\SetEnumerateShortLabel{I}{\textup{(\Roman*)}}

\usepackage[initials,lite,alphabetic]{amsrefs} 

\usepackage[all]{xy}

\theoremstyle{plain}

\newtheorem{thm}{Theorem}[section]
\newtheorem*{lm*}{Lemma}

\newtheorem*{proper*}{Property}
\newtheorem*{thm*}{Theorem}

\newtheorem{conj}[thm]{Conjecture}

\newtheorem{cor}[thm]{Corollary}

\newtheorem{lem}[thm]{Lemma}

\newtheorem{prop}[thm]{Proposition}
\newtheorem{ques}[thm]{Question}

\theoremstyle{definition}

\newtheorem*{df*}{Definition}

\newtheorem{ex-notn}[thm]{Example/Notation}

\newtheorem{exam}[thm]{Example}

\theoremstyle{remark}

\newtheorem*{acknowledgement*}{Acknowledgement}

\newtheorem*{ex*}{Example}
\newtheorem*{exer*}{Exercise}

\newtheorem*{prob*}{Problem}
\newtheorem*{prop*}{Proposition}
\newtheorem*{rem*}{Remark}

\newtheorem{rem}[thm]{Remark}

\def\depth{\operatorname{depth}}

\def\sdepth{\operatorname{sdepth}} 

\def\KK{\mathbb{K}}

\def\ZZ{\mathbb{Z}}

\def\calP{\mathcal{P}}

\def\bdx{\mathbf{x}}

\def\ceil#1{\left\lceil #1 \right\rceil}

\def\floor#1{\left\lfloor #1 \right\rfloor}

\def\Index#1{\emph{#1}}

\def\mySet#1{\left\{ #1\right \}}

\def\sqr#1#2{{\vcenter{\hrule height.#2pt
\hbox{\vrule width.#2pt height#1pt \kern#1pt
\vrule width.#2pt}
\hrule height.#2pt}}}

\let\oldmarginpar\marginpar
\renewcommand\marginpar[1]{\-\oldmarginpar[\raggedleft\footnotesize #1]%
{\raggedright\footnotesize #1}}

\def\opn#1#2{\def#1{\operatorname{#2}}} 
\opn\lex{lex}
\opn\rev{rev}
\opn\Lex{Lex}
\opn\GL{GL}
\opn\initial{in}

\def\Andre{{Andr\'e}}

\def\Cimpoeas{Cimpoea\c{s}}

\begin{document}
\title{When will the Stanley depth increase}
\author{Yi-Huang Shen} 
\address{Department of Mathematics, University of Science and Technology of China, Hefei, Anhui, 230026, China}
\address{Wu Wen-Tsun Key Laboratory of Mathematics, USTC, Chinese Academy of Sciences, Hefei, Anhui, 230026, China}
\thanks{The author is supported by the National Natural Science Foundation of China}
\email{yhshen@ustc.edu.cn} 
\subjclass[2010]{ Primary
05E45, 
05E40, 
06A07; 
Secondary
13C13, 
05C70} 
\keywords{Stanley depth; Squarefree monomial ideal} 
\begin{abstract}
  Let $I\subset S=\KK[x_1,\dots,x_n]$ be an ideal generated by squarefree monomials of degree $\ge d$. If the number of degree $d$ minimal generating monomials $\mu_d(I)\le \min(\binom{n}{d+1},\sum_{j=1}^{n-d}\binom{2j-1}{j})$, then the Stanley depth $\sdepth_S(I)\ge d+1$.
\end{abstract}
\maketitle

\section{Introduction}

Throughout this paper, let $\KK$ be a field and $S=\KK[x_1,\dots,x_n]$ a polynomial ring in $n$ variables over $\KK$. The ring $S$ has a natural $\ZZ^n$-grading. If $M$ is a finitely generated $\ZZ^n$-graded $S$-module, a \Index{Stanley decomposition} of $M$ is a finite direct sum decomposition
\[
\calP: M=\bigoplus_{i=1}^m u_i \KK[Z_i]
\]
of $M$ as a $\ZZ^n$-graded $\KK$-vector space, where each $u_i\in M$ is homogeneous and $Z_i\subset \Set{x_1,\dots,x_n}$. Here, $u_i\KK[Z_i]$ is considered as a free $\KK[Z_i]$-submodule of $M$. The \Index{Stanley depth} of this decomposition is $\sdepth(\calP)=\min\mySet{\left| Z_i \right|: 1\le i \le m}$ and the \Index{Stanley depth} of the module $M$ is
\[
\sdepth(M):=\max\Set{\sdepth(\calP):\text{$\calP$ is a Stanley decomposition of $M$}}.
\]
If we consider isomorphism instead of equality in the previous Stanley decomposition $\calP$, we will land up in the notion of \Index{Hilbert depth}, which is the main topic of \cite{bruns-2009}.

The driving force for investigating the Stanley depth of a finitely generated $\ZZ^n$-graded module $M$ is the conjecture raised by Stanley \cite{MR666158}, which says
\[
\sdepth(M)\ge \depth(M).
\tag{\dag}
\]
This conjecture will imply (\cite[4.5]{MR2366164}) that Cohen-Macaulay simplicial complexes are partitionable, which was separately conjectured by Garsia {\cite[5.2]{MR597728}} and Stanley \cite[p.\ 149]{MR526314}.

To have an insight into the properties of Stanley depth, one lacks the many powerful tools as those for the normal algebraic depth. Deciding the Stanley depth of interesting modules is already a headache for researchers. Currently, the Stanley depth is known only for a very narrow scope of modules, the overwhelming majority of which has equality in the Stanley conjecture (\dag).

The paper \cite{arxiv.0712.2308} by Herzog, Vladoiu and Zheng was a breakthrough along this line. Their method attacks the problem of computing the Stanley depth $\sdepth(I/J)$ for monomial ideals $J\subset I$ in $S$. This method, though not a panacea, contributes fundamentally to the knowledge of Stanley decompositions from both theoretical and computational perspectives. For instance, based on this method, Bir\'o et al.\ \cite{BHKTY} can show that $\sdepth_S(\braket{x_1,\dots,x_n})=\ceil{\frac{n}{2}}$. Notice that $\depth_S(\braket{x_1,\dots,x_n})=1$. Other nontrivial computations and estimates can be found in, for instance, \cite{KeYo2009}, \cite{Okazaki2009}, \cite{arXiv.org:0805.4461} and their references.

Throughout this paper, $I$ will be a monomial ideal in $S$, generated by squarefree monomials of degree $\ge d$. The task of the current paper is to investigate when will $\sdepth(I)\ge d+1$. Our main result is the following theorem.

\begin{thm}
  \label{main-thm}
  Let $I\subset S=\KK[x_1,\dots,x_n]$ be an ideal generated by squarefree monomials of degree $\ge d$. If the number of degree $d$ minimal generating monomials
  \[
  \mu_d(I)\le \min\left(\binom{n}{d+1}, \sum_{j=1}^{n-d}\binom{2j-1}{j}\right),
  \]
  then the Stanley depth $\sdepth_S(I)\ge d+1$.
\end{thm}

Let us finish this introduction by going over the structure of this paper. In section 2, we will go over Herzog, Vladoiu and Zheng's method for computing the Stanley depth of monomial ideals. We will tailor it to the squarefree case and prove a special case of the main theorem. In section 3, we will inspect several combinatorial constructions, which are essential for deciding when will the Stanley depth increase. In the final section, we will complete the proof and provide additional remarks and questions.

\section{Herzog, Vladoiu and Zheng's method}

By convention, we denote the set $\Set{1,2,\dots,n}$ by $[n]$. For the squarefree monomial ideal $I$, consider the associated set
\[
P_I:=\Set{ \Set{i_1,\dots,i_m}\subset [n] :  x_{i_1}\cdots x_{i_m}\in I, 1\le m \le n}.
\]
This is a partially ordered set (poset) with respect to inclusion. When $A,B\in P_I$, the interval $[A,B]$ is the set $\Set{C\in P_I: A\subset C \subset B}$.
Herzog, Vladoiu and Zheng's method \cite[2.5]{arxiv.0712.2308} for squarefree monomial ideals can be easily checked to be equivalent to the following characterization:

\begin{lem}
  Let $k$ be a positive integer. Then $\sdepth(I)\ge k$ if and only if $P_I$ has a disjoint partition $\calP: P_I = \bigcup_{i=1}^l [A_i,B_i]$ such that the cardinalities $\left| B_i \right|\ge k$, $1\le i \le l$.
\end{lem}

This can be further simplified. Consider the reduced associated poset
\[
P_I^k := \Set{\Set{i_1,\dots,i_m}\in P_I: m\le k}.
\]
It is \Index{partitionable} if $P_I^k$ has a disjoint partition $\calP: P_I^k = \bigcup_{i=1}^l [A_i,B_i]$ such that the cardinalities $\left| B_i \right| = k$, $1\le i \le l$. Thus, the previous lemma is equivalent to saying
\[
\sdepth(I)\ge k \Leftrightarrow P_I^k \text{ is partitionable.} \tag{\ddag}
\]
It follows that if $I$ is generated by squarefree monomials of degree $\ge d$, then $\sdepth(I)\ge d$. Meanwhile, we also have $\depth(I)\ge d$ in this case; see \cite[1.3, 3.1]{arxiv.0712.2308}.

\begin{rem}
  \label{sdepth_pure}
  Let $I$ be an $S$-ideal generated by squarefree monomials of degree $\ge d$ and $I_d$ the subideal generated by the degree $d$ generators of $I$. Since $P_I^{d+1}$ differs from $P_{I_d}^{d+1}$ only in the degree $d+1$ part,   $\sdepth(I)\ge d+1$ if and only if $\sdepth(I_d)\ge d+1$. Using the observation $\sdepth(I),\sdepth(I_d)\ge d$, this is equivalent to saying that $\sdepth(I)=d$ if and only if $\sdepth(I_d)=d$. Thus, in the following, we may assume that $I=\braket{I_d}$; we will say $I$ is \Index{pure of degree $d$} in this case.
\end{rem}

\begin{rem}
  \label{sdepth_inclusion}
  Let $I\subset J$ be two $S$-ideals which are generated by squarefree monomials of degree $d$. If $\sdepth(J)\ge d+1$, then $P_J^{d+1}$ is partitionable by (\ddag). The restriction of such a partition to $P_I^{d+1}$ shows that $P_I^{d+1}$ is also partitionable. Thus, $\sdepth(I)\ge d+1$.  Notice that in general, we cannot compare $\sdepth(I)$ with $\sdepth(J)$ even  if there exists containment between $I$ and $J$. For instance, for the three squarefree monomial ideals $I_1:=\braket{1} \supset I_2:=\braket{x_1,x_2} \supset I_3:=\braket{x_1}$ in $S=\KK[x_1,x_2]$, we will have $\sdepth(I_1)=\sdepth(I_3)=2>\sdepth(I_2)=1$.
\end{rem}

\begin{cor}
  \label{NC-1}
  Suppose $I$ is generated by squarefree monomials of degree $\ge d$ and $\sdepth(I)\ge d+1$, then the number of degree $d$ minimal generators $\mu_d(I)\le \binom{n}{d+1}$.
\end{cor}

\begin{proof}
  Since $\sdepth(I)\ge d+1$, $P_I^{d+1}$ is partitionable and has a partition $P_{I}^{d+1} = \bigcup_{i=1}^l [A_i,B_i]$ with $\left| B_i \right|=d+1$. Now $\binom{n}{d+1} \ge \left| \Set{B_i: 1\le i \le l} \right| \ge \left| \mySet{B_i: \left| A_i \right|=d} \right|=\left| \mySet{A_i: \left| A_i \right|=d} \right|=\mu_d(I)$.
\end{proof}

When $n\ge 2d+1$, we have $\mu_d(I)\le \binom{n}{d} \le \binom{n}{d+1}$. Thus Corollary \ref{NC-1}  does not provide much information in this case. However, we have

\begin{prop}
  \label{small_case}
  Suppose $n\ge 2d+1$ and $I$ is generated by squarefree monomials of degree $\ge d$. Then $\sdepth(I) \ge d+1$.
\end{prop}

\begin{proof}
  Recall that the squarefree Veronese ideal $I_{n,d}$ is the ideal generated by all degree $d$ squarefree monomials of $S=\KK[x_1,\dots,x_n]$. It follows from \cite[1.1]{project4} that $\sdepth(I_{n,d})\ge d+1$. Now, we use Remarks \ref{sdepth_pure} and \ref{sdepth_inclusion}.
\end{proof}

Inspired by the proof of Proposition \ref{small_case}, we raise the following conjecture on the Stanley depth of squarefree monomial ideals:

\begin{conj}
  If $I \subset S=\KK[x_1,\dots,x_n]$ is an ideal generated by squarefree monomials of degree $\ge d$, then
  \[
  \sdepth_S(I)\ge d+ \floor{\binom{n}{d+1}\Big/\binom{n}{d}}.
  \]
\end{conj}

Thanks to \cite[2.2]{project4}, we know $\sdepth(I_{n,d})\le d+ \floor{\binom{n}{d+1}\big/\binom{n}{d}}$. Hence this conjecture is stronger than the special case \cite[2.4]{project4}, which is also separately conjectured by {\Cimpoeas} in \cite[1.6]{arXiv:0907.1232}.

\section{The associated pure complex}

Let $k$ be a positive integer. By \cite[4.2.6]{MR1251956}, any integer $x\ge 1$ can be written uniquely in the form
\begin{equation}
  \label{MacaulayRepr}
  x = \binom{a_k}{k}+\binom{a_{k-1}}{k-1}+\cdots + \binom{a_i}{i}
\end{equation}
such that $a_k> \cdots > a_i \ge i >0$. The above sum is called the $k$-th  \Index{Macaulay representation} of $x$. For any integer $j$, $1\le j < i$, we set $a_j=j-1$. We shall call $a_k,a_{k-1},\dots,a_1$ the $k$-th \Index{Macaulay coefficients} of $x$.  They have the following nice property.

\begin{lem}
  [{\cite[4.2.7]{MR1251956}}]
  \label{427}
  Let $a_k,\dots,a_1$, respectively $a_k',\dots,a_1'$ be the $k$-th Macaulay coefficients of $x$, respectively $x'$. Then $x>x'$ if and only if 
  \[
  (a_k,\dots,a_1) > (a_k',\dots,a_1')
  \]
  in the lexicographical order.
\end{lem}

Now, let $\delta:=n-d$ be the difference of degrees and write $\xi_{\delta}:=\sum_{j=1}^{\delta}\binom{2j-1}{j}$.  In Theorem \ref{main-thm}, we need comparing the integer $\xi_{n-d}$ with $\binom{n}{d+1}$. 

\begin{lem}
  \label{min}
  Let $1\le \delta < n$. Then
  \[
  \min \left(\xi_{\delta}, \binom{n}{\delta-1}\right) =
  \begin{cases}
    \xi_{\delta} &  \text{ if $n\ge 2\delta$}, \\
    \binom{n}{\delta-1} & \text{ if $n\le 2\delta-1$}.
  \end{cases}
  \]
\end{lem}

\begin{proof}
  The cases when $\delta=1$ and $2$ can be verified directly. Thus, we assume that $\delta\ge 3$.  Note that the $(\delta-1)$-th Macaulay coefficients of $\binom{n}{\delta-1}$ is $n,\delta-3,\delta-4,\dots,1,0$. Meanwhile, the $(\delta-1)$-th Macaulay coefficients of $\xi_{\delta}$ is $2\delta-1,2\delta-3,\dots,7,5,4$.  When $n\ge 2\delta$,
  \[
  (n,\delta-3,\delta-4,\dots,1,0) > (2\delta-1,2\delta-3,\dots,7,5,4)
  \]
  in the lexicographical order. When $n\le 2\delta-1$, we have the opposite comparison result. Therefore, the conclusion follows from Lemma \ref{427}.
\end{proof}

For each squarefree monomial $m=x_{i_1}\cdots x_{i_k}\in S$, we denote the set $\Set{1,\dots,n}\setminus \Set{i_1,\dots,i_k}$ by $m^\complement$. Suppose $I$ is a squarefree monomial $S$-ideal and $G(I)$ is the set of minimal generating monomials of $I$. We call the simplicial complex $\Delta^\complement(I):=\braket{m^\complement: m\in G(I)}$ the \Index{complement complex of $I$}. For each simplicial complex over $[n]$, there is a unique squarefree monomial ideal $I$ such that $\Delta=\Delta^\complement(I)$. Thus, we will call $I$ the \Index{complement ideal of $\Delta$}. It is clear that $I$ is generated by its degree $k$ part $I_k$ if and only if $\Delta^\complement(I)$ is pure of dimension $n-k-1$. When $\Delta^\complement(I)$ is pure, the number of facets $f_{n-k-1}(\Delta^\complement(I))=\mu(I)$.

Now, let $I$ be a squarefree monomial ideal which is pure of degree $d$. We will relate the reduced associated poset $P_I^{d+1}$ of $I$ with its complement complex $\Delta^\complement(I)$. Each interval $[A,B]\subset P_I^{d+1}$ with $\left| A \right|=d$ and $\left| B \right|=d+1$ corresponds to the pair $([n]\setminus A, [n]\setminus B)$. Notice that $[n]\setminus A$ is a facet of $\Delta^\complement(I)$ and $[n]\setminus B$ is a face contained in $[n]\setminus A$. Now it is clear that the following three conditions are equivalent:
\begin{enumerate}[a]
  \item The Stanley depth $\sdepth(I)\ge d+1$;
  \item The reduced associated poset $P_I^{d+1}$ is partitionable;
  \item For each facet $F$ of $\Delta=\Delta^\complement(I)$, we can suitably drop a vertex to get a face $\widetilde{F}$, such that all these $\widetilde{F}$'s are pairwise distinct.
\end{enumerate}

The third condition is closely related to the problem of finding systems of distinct representatives (SDR). It provides the framework for our further investigation. In the following, we will call a pure simplicial complex $\Delta$  \Index{uniformly collapsible} if it satisfies the third condition above. It is straightforward to see that if $\Delta$ is a uniformly collapsible complex of dimension $\delta-1$, then $f_{\delta-2}\ge f_{\delta-1}$. Here, $f(\Delta)=(f_{-1}=1,f_0,\dots,f_{\delta-1})$ is the $f$-vector of $\Delta$. Actually, we have the following characterization:

\begin{lem}
  \label{UC}
  For any $(\delta-1)$-dimensional pure simplicial complex $\Delta$, the following two conditions are equivalent:
  \begin{enumerate}[a]
    \item The complex $\Delta$ is uniformly collapsible;
    \item For each $(\delta-1)$-dimensional (pure) subcomplex $\Delta'$, we have $f_{\delta-2}(\Delta') \ge f_{\delta-1}(\Delta')$.
  \end{enumerate}
\end{lem}

\begin{proof}
  For the pure complex $\Delta$, we consider its associated bipartite graph $G$ defined as follows. The vertex set is $V(G)=X\cup Y$ where $X$ is the set of all ($\delta-1$)-dimensional  faces (facets) of $\Delta$, while $Y$ is the set of all ($\delta-2$)-dimensional  faces of $\Delta$. An edge of $G$ has endpoints $x\in X$ and $y\in Y$ if and only if $x\supset y$ in $\Delta$. We will use $\Gamma(x)$ to denote the set of all vertices adjacent to a given vertex $x\in X$. If $A$ is a subset of $X$, we denote by $\Gamma(A)$ the set $\bigcup_{a\in A}\Gamma(a)$. Let $\Delta'(A)$ be the simplicial complex $\braket{A}$. Then $\Gamma(A)$ is the set of all ($\delta-2$)-dimensional  faces of $\Delta'(A)$. Now, our claim follows directly from the famous P.\ Hall's marriage theorem \cite[5.1]{MR1871828}, which says that a necessary and sufficient condition for there to be a complete matching from $X$ to $Y$ in $G$ is that $|\Gamma(A)| \ge |A|$ for every $A \subset X$. Since $f_{\delta-2}(\Delta'(A))=|\Gamma(A)|$ and $f_{\delta-1}(\Delta'(A))=|A|$, we are done.
\end{proof}

\begin{cor}
  \label{NUC}
  The pure simplicial complex $\Delta$ is uniformly collapsible if and only if $f_{\delta-2}(\Delta) \ge f_{\delta-1}(\Delta)$ and all its proper subcomplexes of same dimension $\delta-1$ are uniformly collapsible .
\end{cor}

Before we proceed to the next technical lemma, we need reviewing one nice combinatorial interpretation of the Catalan numbers $C_n:=\frac{1}{n+1}\binom{2n}{n}=\binom{2n}{n}-\binom{2n}{n+1}$ for $n\ge 0$.

\begin{rem}
  [{\cite[14.8]{MR1871828}}]
  \label{paths}
  Consider walks in the $X$-$Y$ plane where each step is $U:(x,y)\to (x+1,y+1)$ or $D:(x,y)\to (x+1,y-1)$. Let $A=(0,k)$ and $B=(n,m)$ be two integral points on the upper halfplane.  It follows from the {\Andre}'s \emph{reflection principle} that there are $\binom{n}{l_2}-\binom{n}{l_1}$ paths from $A$ to $B$ that do not meet the $X$-axis. Here, $2l_1=n-k-m$ and $2l_2=n-m+k$. As a result, there are $C_{n-1}$ paths from $(0,0)$ to $(2n,0)$ in the upper halfplane that do not meet the $X$-axis between these two points. Furthermore, if we allow the paths to meet the $X$-axis without crossing, then the number is $C_n$.
\end{rem}

With respect to the Macaulay representation \eqref{MacaulayRepr}, we define
\[
\partial_{k-1}(x) = \binom{a_k}{k-1}+\binom{a_{k-1}}{k-2} +\cdots + \binom{a_i}{i-1}.
\]

\begin{lem}
  \label{key-lemma}
  For any positive integer $x$ such that $x\le \xi_k:= \sum_{j=1}^k \binom{2j-1}{j}$, we have $\partial_{k-1}(x)\ge x$.
\end{lem}

\begin{proof}
  Suppose \eqref{MacaulayRepr} gives the Macaulay representation of $x$. We need to show
  \[
  \sum_{j=i}^k \binom{a_j}{j-1} \ge \sum_{j=i}^k \binom{a_j}{j}.
  \]
  In view of Lemma \ref{427}, we obtain $a_k\le 2k-1$. If $a_k=2k-1$, we can consider the case where $k'=k-1$ and $x'=x-\binom{2k-1}{k}$. Now $x'\le \sum_{j=1}^{k'-1}\binom{2j-1}{j}$. The conclusion will follow from the induction on $k$, with the case $k=1$ being trivial.

  Thus we may assume that $a_k < 2k-1$. Let $k_0$ be the smallest integer such that for all $k_0\le j \le k$ we have $a_j< 2j-1$.  Now, it suffices to prove
  \begin{equation}
    \label{aim}
    \sum_{j=k_0}^k \left( \binom{a_j}{j-1}-\binom{a_j}{j} \right)
    \ge
    \sum_{j=i}^{k_0-1} \left(\binom{a_j}{j}-\binom{a_j}{j-1} \right).
  \end{equation}

  First of all, let us look at the summand on the left hand side of the inequality \eqref{aim}.  By our choice of $k_0$, we have $k_0>1$ and  $a_{k_0-1}\ge 2k_0-3$. Thus, for $j=k_0,k_0+1,\dots,k$, we have $a_j\ge j+k_0-2$.  When $a_j<2j-1$, the integer
  \begin{equation}
    \label{func1}
    \binom{a_j}{j-1}-\binom{a_j}{j}=\binom{a_j}{a_j-j+1}-\binom{a_j}{a_j-j}
  \end{equation}
  is the number of paths in the $X$-$Y$ plane from $A=(0,1)$ to $B_{j,a_j}=(a_j,2j-1-a_j)$ that do not meet the $X$-axis. In particular, this is a positive integer. When $a_j<2j-2$, any such a path followed by a step $D$ as in Remark \ref{paths} gives a path from $A$ to $B_{j,a_j+1}$. Thus, \eqref{func1} is an increasing function for $a_j\in\Set{j+k_0-2,j+k_0-1,\dots,2j-2}$. Now the infimum of the left hand side of \eqref{aim} is achieved when $a_j=j+k_0-2$.  Henceforth, without loss of generality, we may assume that $k=k_0$ and $a_k=2k-2$, whence $a_{k-1}=2k-3$.

  Next, let us consider the summand on the right hand side of the inequality \eqref{aim}. Notice that $a_k=2k-2$, thus $a_j\le k-2+j$. Now we have
  \begin{equation}
    \label{func2}
    \binom{a_j}{j}-\binom{a_j}{j-1} = \binom{a_j}{j -1}\left(\frac{a_j-j+1}{j}-1 \right),
  \end{equation}
  which is positive only when $a_j\ge 2j-1$. When this condition is indeed satisfied, the integer \eqref{func2} is the number of paths in the $X$-$Y$ plane from $A=(0,1)$ to $B_{j,a_j}=(a_j,a_j+1-2j)$ that do not meet the $X$-axis.  Any such a path followed by a step $U$ as in Remark \ref{paths} gives a path from $A$ to $B_{j,a_j+1}$.  Thus, \eqref{func2} is an increasing function for $a_j\in \Set{2j-1,2j,\dots,k-2+j}$.  Now the supremum of the right hand side of \eqref{aim} is achieved when $i=1$ and $a_j=k-2+j$ for $j=1,\dots,k-1$.

  Now it suffices to prove
  \[
  \binom{2k-2}{k-1}-\binom{2k-2}{k}
  \ge
  \sum_{j=1}^{k-1} \left( \binom{k-2+j}{j}-\binom{k-2+j}{j-1} \right).
  \]
  As a matter of fact, we have
  \begin{align*}
    LHS-RHS =& \sum_{j=1}^k \binom{k-2+j}{j-1} - \sum_{j=1}^k \binom{k-2+j}{j} \\
    = & \sum_{j=2}^k \left(\binom{k-1+j}{j-1}-\binom{k-2+j}{j-2}\right) + \binom{k-1}{0} \\
    & \qquad - \sum_{j=1}^k \left(\binom{k-1+j}{j}-\binom{k-2+j}{j-1}\right) \\
    = & \left(\binom{2k-1}{k-1}-\binom{k}{0}\right) + \binom{k-1}{0} - \left(\binom{2k-1}{k}-\binom{k-1}{0}\right) \\
    = & 1.
  \end{align*}
  One can also explain this difference being $1$ by the paths argument in Remark \ref{paths}.
\end{proof}

Next, consider the following property $(*)$:
\begin{quote}
  If $\Delta$ is a pure simplicial complex of dimension $\delta-1$ and $f_{\delta-1}(\Delta) \le f_{\delta-2}(\Delta)$, then $\Delta$ is uniformly collapsible.
\end{quote}
For investigating this property, we have to be equipped with further apparatus.  We will need the following fact from \cite[p79]{MR1286816}. Define the reverse lexicographical order $\le_{rlex}$ on the $k$-subsets of $[n]:=\Set{1,2,\dots,n}$ as follows. Let $S=\Set{{i_1}<\cdots < {i_k}}$ and $T=\Set{{j_1}<\cdots < {j_k}}$ be two $k$-subsets. We say $S<_{rlex} T$ if for some $q$, we have $i_q < j_q$ and $i_p=j_p$ for $p>q$. A collection $C$ of $k$-subsets of $[n]$ is \Index{compressed} if $S<_{rlex} T$ and $T\in C$ imply $S\in C$. Since $\le_{rlex}$ is a total ordering, there is only one compressed collection of $k$-subsets of size $l$, $1\le l \le \binom{n}{k}$. We will call it $C_{n,k}^l$ and denote the ($k-1$)-dimensional simplicial complex $\braket{C_{n,k}^l}$ by $\Delta_l^{n,k}$. The complement ideal of $\Delta_l^{n,k}$ will be written as $I_{n,n-k}^l$. It is generated by $l$ squarefree monomials of degree $n-k$. For $1\le d \le n$ and $l=\binom{n}{d}$, the ideal $I_{n,d}^l$ is the usual squarefree Veronese ideal $I_{n,d}$. 

The \Index{shadow} of any collection $C$ of $k$-subsets is
\[
\partial C=\Set{S: \left| S \right|=k-1, S\subset T \text{ for some $T\in C$}}.
\]
The shadow $\partial C_{n,k}^l$ is also compressed and $\left| \partial C_{n,k}^l \right| = \partial_{k-1}(l)$. The proof of this fact can be found, for instance, in \cite[Section 8]{MR513002}. This implies that $f_{k-2}(\Delta_l^{n,k})=\partial_{k-1}(f_{k-1}(\Delta_{l}^{n,k}))=\partial_{k-1}(l)$.

When $\Delta$ is pure of dimension $\delta-1$ and $C$ is the set of all facets, then $\partial C$ is the set of all $(\delta-2)$ faces. In general, we will have $f_{\delta-2}(\Delta)\ge \partial_{\delta-1}(f_{\delta-1}(\Delta))$, namely $\left| \partial C \right| \ge \partial_{\delta-1}(\left| C \right|)$; see \cite[8.1]{MR513002}.

\begin{exam}
  \label{NotUC}
  The simplicial complex $\Delta=\Delta_{\xi_\delta+1}^{n,\delta}$ is not uniformly collapsible. For this, it suffices to observe that $f_{\delta-1}(\Delta)=\xi_{\delta}+1=\binom{2}{1}+\sum_{j=2}^{\delta}\binom{2j-1}{j}$. Thus $f_{\delta-2}(\Delta)=\partial_{\delta-1}(\xi_\delta+1)=\binom{2}{0}+\sum_{j=2}^{\delta}\binom{2j-1}{j-1} = \xi_{\delta}$ and $f_{\delta-1}(\Delta)> f_{\delta-2}(\Delta)$. Now apply Corollary \ref{NUC}.
\end{exam}

If we combine Corollary \ref{NUC} with Lemma \ref{key-lemma}, we obtain the following result:

\begin{cor}
  If $f_{\delta-1}(\Delta)=\xi_{\delta}+1$, then the property $(\ast)$ holds.
\end{cor}

However, the property ($\ast$) does not hold in general. 

\begin{exam}
  \label{counter-example}
  We already know that the simplicial complex $\Delta_{\xi_{\delta}+1}^{n,\delta}$ over the vertex set $[n]$ is not uniformly collapsible.  Now, let $\widetilde{\Delta}=\braket{\Delta_{\xi_{\delta}+1}^{n,\delta} , \Set{n,n+1,\dots,n+\delta-1}}$ be a new simplicial complex over the vertex set $[n+\delta-1]$. It is again pure of dimension $\delta-1$. Notice that
  \[
  f_{\delta-1}(\widetilde{\Delta})=f_{\delta-1}(\Delta)+1=\xi_{\delta}+2
  \]
  and when $\delta\ge 3$
  \[
  f_{\delta-2}(\widetilde{\Delta}) = f_{\delta-2}(\Delta_{\xi_{\delta}+1}^{n,\delta}) + f_{\delta-2} (\braket{\Set{n,n+1,\dots,n+\delta-1}}) = \xi_{\delta} +\delta.
  \]
  Whence, we have $f_{\delta-2}(\widetilde{\Delta}) > f_{\delta-1}(\widetilde{\Delta})$. However, $\widetilde{\Delta}$ is not uniformly collapsible because of the existence of the pure subcomplex $\Delta_{\xi_{\delta}+1}^{n,\delta}$.
\end{exam}

In the current context, we always assume that $n/2\le d\le n$, whence $2\delta\le n$. The obstacle in the previous example is created by introducing extra vertices; now the number of vertices is at least $3\delta-1$. Thus, we are interested in the following question:

\begin{ques}
  Fix the degree difference $\delta$.  If $n=2\delta$, does the property $(*)$ hold? If the answer is positive, what is the largest integer $n<3\delta-1$ such that $(*)$ holds?
\end{ques}

\section{Proof of Theorem \ref{main-thm}}

We have gathered all the apparatus for proving the main theorem. 

\begin{proof}
  By virtue of Remark \ref{sdepth_pure}, we may assume that $I$ is pure of degree $d$. For $1\le d <n$, write $\delta=n-d$ for the difference of degrees. 

  When $n\ge 2d+1$, we have $n\le 2\delta-1$. Thus 
  \[
  \min \left(\xi_{n-d}, \binom{n}{d+1}\right) = \binom{n}{d+1} 
  \]
  by virtue of Lemma \ref{min}. The condition $\mu(I)\le \binom{n}{d+1}$ is automatically satisfied and we have $\sdepth(I)\ge d+1$ from Proposition \ref{small_case}. 

  On the other hand, when $1\le d < n \le 2d$, we have $n\ge 2\delta$. Now
  \[
  \min \left(\xi_{n-d}, \binom{n}{d+1}\right) = \xi_{n-d}.
  \]
  If $\mu(I)\le \xi_{\delta}$, its complement complex $\Delta^\complement(I)$ is uniformly collapsible from Lemmas \ref{UC} and \ref{key-lemma}. Thus $\sdepth_S(I)\ge d+1$.
\end{proof}

\begin{rem}
  We want to emphasize that the condition in Theorem \ref{main-thm} is optimal. With $\delta=n-d$, there is not much to mention for the case $n\le 2\delta -1$. When $n\ge 2\delta$, we will take $I=I_{n,d}^{\xi_{\delta}+1}$.  It has been manifested in Example \ref{NotUC} that the complement complex $\Delta_{\xi_{\delta}+1}^{n,\delta}$ is not uniformly collapsible, whence $\sdepth_S(I)=d$. We finish by noticing that $\mu_d(I)=\xi_{\delta}+1$.
\end{rem}

\begin{rem}
  \label{smallest-d}
  When $n/2\le d<n$, the set
  \[
  \Xi:= \Set{I \subset S \mid \text{$I$ is pure of degree $d$ and}\sdepth(I)=d}
  \]
  is non-empty and partially ordered with respect to inclusion.  If $I\in \Xi$ is minimal, then $\mu(I) \ge \xi_{\delta}+1$. This inequality can be strict if the dimension $n$ is not too small relative to the difference $\delta=n-d$. We will only show this in the special case when $d=n-2$.  Let $G$ be the graph on $[n]$ ($1$-dimensional pure simplicial complex) with edges
  \[
  E(G)=\Set{\Set{1,2},\Set{2,3},\dots,\Set{n-1,n},\Set{n,0},\Set{1,3}}.
  \]
  It is a circle with a chord. All $1$-dimensional proper subcomplexes of $G$ are uniformly collapsible, while $G$ itself is not. Let $I$ be the degree $n-2$ complement ideal of the complex $G$. It satisfies that $\sdepth(I)=n-2$ and $\mu(I)=n+1$. Furthermore, this ideal is minimal in $\Xi$.
\end{rem}

Since $n+1$ is smaller when compared with $\binom{n}{d}$ or $\binom{n}{d+1}$ in this situation, we are interested in

\begin{ques}
  What is $ \max\Set{\mu(I)\mid \text{$I$ is minimal in $\Xi$}}$?
\end{ques}

Since $\binom{n}{d} > \binom{n}{d+1}$, the maximal element of $\Xi$ is the squarefree Veronese ideal $I_{n,d}$. Thus, the number
\[
\max\Set{\mu(I)\mid \text{$I$ is maximal in $\Xi$}}
\]
is clear.

\begin{rem}
  \label{largest-d}
  When $n/2\le d<n$, the set
  \[
  \Xi^\complement := \Set{I \subset S \mid \text{$I$ is pure of degree $d$ and}\sdepth(I)>d}
  \]
  is also nonempty.  For any $I\in \Xi^\complement$, we have $\mu(I)\le \binom{n}{d+1}$. We will show that
  \[
  \max\Set{\mu(I) | I\in \Xi^\complement}=\binom{n}{d+1}.
  \]

  Suppose $k$ is an integer with $1\le k \le n-1$. If a squarefree monomial ideal $I$ is pure of degree $k$ and $\sdepth(I)\ge k+1$, we have a union of  disjoint intervals $\bigcup_{\bdx^m\in G(I)} [m,\widetilde{m}]$ in $P_I^{k+1}$, with $|\widetilde{m}|=k+1$ for each $\bdx^m\in G(I)$. Here, $\bdx^m$ stands for $x_{i_1}x_{i_2}\cdots x_{i_k}$ if $m=\Set{i_1,\dots,i_k}$. Now, simply set $J=\braket{\bdx^{\widetilde{m}^\complement} \mid m\in G(I)}$. The squarefree monomial ideal $J$ is pure of degree $n-k-1$ and $\sdepth(J)\ge n-k$. This correspondence from $I$ to $J$, though not one-to-one, preserves the minimal number of generators.

  Now, we are reduced to show the existence of a squarefree monomial ideal $J$ that is pure of degree $n-d-1$ with $\mu(J)=\binom{n}{d+1}$ and $\sdepth(J)\ge n-d$. This monomial ideal $J$ has to be the squarefree Veronese ideal $I_{n,n-d-1}$. Since $2d\ge n-1$, it has the desired properties.
\end{rem}

Notice that any set of squarefree monomials has a squarefree shadow; see \cite[2.2]{MR2583270}. Thus, we can prove Theorem \ref{main-thm} directly, without resorting to the complement complex. However, we find this approach less intuitive, especially during the construction of the simplicial complex $\widetilde{\Delta}$ in Example \ref{counter-example} and the graph $G$ in Remark \ref{smallest-d}.


\begin{bibdiv}
\begin{biblist}

\bib{MR2583270}{incollection}{
      author={Bonanzinga, Vittoria},
      author={Ene, Viviana},
      author={Olteanu, Anda},
      author={Sorrenti, Loredana},
       title={An overview on the minimal free resolutions of lexsegment
  ideals},
        date={2009},
   booktitle={Combinatorial aspects of commutative algebra},
      series={Contemp. Math.},
      volume={502},
   publisher={Amer. Math. Soc.},
     address={Providence, RI},
       pages={5\ndash 24},
}

\bib{MR1251956}{book}{
      author={Bruns, Winfried},
      author={Herzog, J{\"u}rgen},
       title={Cohen-{M}acaulay rings},
      series={Cambridge Studies in Advanced Mathematics},
   publisher={Cambridge University Press},
     address={Cambridge},
        date={1993},
      volume={39},
        ISBN={0-521-41068-1},
}

\bib{BHKTY}{article}{
      author={Bir{\'o}, Csaba},
      author={Howard, David~M.},
      author={Keller, Mitchel~T.},
      author={Trotter, William~T.},
      author={Young, Stephen~J.},
       title={Interval partitions and {S}tanley depth},
        date={2010},
        ISSN={0097-3165},
     journal={J. Combin. Theory Ser. A},
      volume={117},
       pages={475\ndash 482},
         url={http://dx.doi.org/10.1016/j.jcta.2009.07.008},
}

\bib{bruns-2009}{article}{
      author={Bruns, Winfried},
      author={Krattenthaler, Christian},
      author={Uliczka, Jan},
       title={Stanley decompositions and {H}ilbert depth in the {K}oszul
  complex},
        date={2010},
        ISSN={1939-0807},
     journal={J. Commut. Algebra},
      volume={2},
       pages={327\ndash 357},
}

\bib{arXiv:0907.1232}{article}{
      author={Cimpoea\c{s}, Mircea},
       title={Stanley depth of square free {V}eronese ideals},
        date={2009},
      eprint={arXiv:0907.1232},
         url={http://www.citebase.org/abstract?id=oai:arXiv.org:0907.1232},
}

\bib{MR1286816}{article}{
      author={Duval, Art~M.},
       title={A combinatorial decomposition of simplicial complexes},
        date={1994},
        ISSN={0021-2172},
     journal={Israel J. Math.},
      volume={87},
       pages={77\ndash 87},
         url={http://dx.doi.org/10.1007/BF02772984},
}

\bib{MR597728}{article}{
      author={Garsia, Adriano~M.},
       title={Combinatorial methods in the theory of {C}ohen-{M}acaulay rings},
        date={1980},
        ISSN={0001-8708},
     journal={Adv. in Math.},
      volume={38},
       pages={229\ndash 266},
         url={http://dx.doi.org/10.1016/0001-8708(80)90006-7},
}

\bib{MR513002}{incollection}{
      author={Greene, Curtis},
      author={Kleitman, Daniel~J.},
       title={Proof techniques in the theory of finite sets},
        date={1978},
   booktitle={Studies in combinatorics},
      series={MAA Stud. Math.},
      volume={17},
   publisher={Math. Assoc. America},
     address={Washington, D.C.},
       pages={22\ndash 79},
}

\bib{MR2366164}{article}{
      author={Herzog, J{\"u}rgen},
      author={Jahan, Ali~Soleyman},
      author={Yassemi, Siamak},
       title={Stanley decompositions and partitionable simplicial complexes},
        date={2008},
        ISSN={0925-9899},
     journal={J. Algebraic Combin.},
      volume={27},
       pages={113\ndash 125},
}

\bib{arxiv.0712.2308}{article}{
      author={Herzog, J{\"u}rgen},
      author={Vladoiu, Marius},
      author={Zheng, Xinxian},
       title={How to compute the {S}tanley depth of a monomial ideal},
        date={2009},
        ISSN={0021-8693},
     journal={J. Algebra},
      volume={322},
       pages={3151\ndash 3169},
         url={http://dx.doi.org/10.1016/j.jalgebra.2008.01.006},
}

\bib{project4}{article}{
      author={Keller, Mitchel~T.},
      author={Shen, Yi-Huang},
      author={Streib, Noah},
      author={Young, Stephen~J.},
       title={On the {S}tanley depth of squarefree {V}eronese ideals},
        date={2011},
        ISSN={0925-9899},
     journal={J. Algebraic Combin.},
      volume={33},
       pages={313\ndash 324},
         url={http://dx.doi.org/10.1007/s10801-010-0249-1},
}

\bib{KeYo2009}{article}{
      author={Keller, Mitchel~T.},
      author={Young, Stephen~J.},
       title={Stanley depth of squarefree monomial ideals},
        date={2009},
        ISSN={0021-8693},
     journal={J. Algebra},
      volume={322},
       pages={3789\ndash 3792},
         url={http://dx.doi.org/10.1016/j.jalgebra.2009.05.021},
}

\bib{Okazaki2009}{article}{
      author={Okazaki, Ryota},
       title={A lower bound of {S}tanley depth of monomial ideals},
        date={2011},
        ISSN={1939-0807},
     journal={J. Commut. Algebra},
      volume={3},
       pages={83\ndash 88},
}

\bib{arXiv.org:0805.4461}{article}{
      author={Shen, Yi-Huang},
       title={Stanley depth of complete intersection monomial ideals and
  upper-discrete partitions},
        date={2009},
        ISSN={0021-8693},
     journal={J. Algebra},
      volume={321},
       pages={1285\ndash 1292},
}

\bib{MR526314}{article}{
      author={Stanley, Richard~P.},
       title={Balanced {C}ohen-{M}acaulay complexes},
        date={1979},
        ISSN={0002-9947},
     journal={Trans. Amer. Math. Soc.},
      volume={249},
       pages={139\ndash 157},
         url={http://dx.doi.org/10.2307/1998915},
}

\bib{MR666158}{article}{
      author={Stanley, Richard~P.},
       title={Linear {D}iophantine equations and local cohomology},
        date={1982},
        ISSN={0020-9910},
     journal={Invent. Math.},
      volume={68},
       pages={175\ndash 193},
}

\bib{MR1871828}{book}{
      author={van Lint, J.~H.},
      author={Wilson, R.~M.},
       title={A course in combinatorics},
     edition={Second},
   publisher={Cambridge University Press},
     address={Cambridge},
        date={2001},
        ISBN={0-521-00601-5},
}

\end{biblist}
\end{bibdiv}


\end{document}